\documentclass[12pt]{amsart}
\usepackage{enumerate, url, amssymb, yhmath, nicefrac, mathrsfs, graphicx, pdfsync, color}
\newtheorem{theorem}{Theorem}[section]
\newtheorem{lemma}[theorem]{Lemma}
\newtheorem*{lemma*}{Lemma}

\numberwithin{equation}{section}

\theoremstyle{definition}
\newtheorem{definition}[theorem]{Definition}

\theoremstyle{remark}

\numberwithin{equation}{section}

\newcommand{\abs}[1]{\lvert#1\rvert}

\newcommand{\C}{\mathbb{C}}
\newcommand{\D}{\partial}

\newcommand{\W}{\mathscr{W}}

\newcommand{\X}{\mathbb{X}}
\newcommand{\U}{\mathbb{U}}
\newcommand{\V}{\mathbb{V}}

\newcommand{\dd}{\mathbb{D}}
\newcommand{\Y}{\mathbb{Y}}

\newcommand{\dtext}{\textnormal d}

\newcommand{\onto}{\xrightarrow[]{{}_{\!\!\textnormal{onto\,\,}\!\!}}}

\newcommand{\bydef}{\stackrel {\textnormal{def}}{=\!\!=} }
\newcommand{\quasi}{\stackrel {\textnormal{quasi}}{=\!\!=\!\!=} }
\DeclareMathOperator{\diam}{diam}

\DeclareMathOperator{\im}{Im}

\DeclareMathOperator{\loc}{loc}

\def\le{\leqslant}
\def\ge{\geqslant}

\def\dd{{\mathbb D}}

\def\cl{{\mathscr L}}

\def\fz{\infty}

\def\loc{{\mathop\mathrm{\,loc\,}}}

\def\ez{\epsilon}

\def\bint{{\ifinner\rlap{\bf\kern.25em--}
\int\else\rlap{\bf\kern.45em--}\int\fi}\ignorespaces}

\def\bbint{{\ifinner\rlap{\bf\kern.25em--}
\hspace{0.078cm}\int\else\rlap{\bf\kern.45em--}\int\fi}\ignorespaces}

\def\diam{{\mathop\mathrm{\,diam\,}}}

\def\r{\right}
\def\lf{\left}

\begin{document}
\title{Singularities in $\mathscr L^{p}$-quasidisks}
\author[T. Iwaniec]{Tadeusz Iwaniec}
\address{Department of Mathematics, Syracuse University, Syracuse,
NY 13244, USA}
\email{tiwaniec@syr.edu}

\author[J. Onninen]{Jani Onninen}
\address{Department of Mathematics, Syracuse University, Syracuse,
NY 13244, USA and Department of Mathematics and Statistics, P.O.Box 35 (MaD) FI-40014 University of Jyv\"askyl\"a, Finland}
\email{jkonnine@syr.edu}

\author[Z. Zhu]{Zheng Zhu}
\address{Department of Mathematics and Statistics, P.O.Box 35 (MaD) FI-40014 University of Jyv\"askyl\"a, Finland}
\email{zheng.z.zhu@jyu.fi}

\thanks{ T. Iwaniec was supported by the NSF grant DMS-1802107.
J. Onninen was supported by the NSF grant DMS-1700274.}

\subjclass[2010]{Primary 30C60; Secondary  30C62}


\keywords{Cusp, mappings of integrable distortion, quasiconformal, quasidisc.}

\maketitle

\begin{abstract}  We study planar domains with  exemplary boundary singularities   of the form of cusps.  A natural question is how much elastic energy is needed to flatten these cusps; that is,  to remove  singularities. We give, in  a   connection of quasidisks,  a sharp integrability condition for the distortion function to answer  this question.
\end{abstract}

\section{Introduction and Overview} The subject matter emerge most clearly when the setting is more general than we actually present it here. Thus we suggest, as a possibility, to consider two planar sets $\,\mathbb X, \mathbb Y\, \subset \mathbb C\,$ of the same global topological configuration, meaning that there is a sense preserving homeomorphism $\, f : \mathbb C \onto \mathbb C\,$ which takes $\,\mathbb X\,$ onto $\mathbb Y\,$. Clearly $\,f : \mathbb C\setminus \mathbb X \onto \mathbb C \setminus \mathbb Y\,$. We choose two examples; one from naturally occurring  Geometric Function Theory (GFT) and the other from mathematical models of Nonlinear Elasticity (NE).  The first one deals with \textit{quasiconformal mappings} $\, f : \mathbb C \onto \mathbb C\,$ and the associated concept of a \textit{quasidisk}, whereas the unexplored perspectives come from NE.
From these perspectives we look at the ambient space $\,\mathbb C\,$  as made of a material whose elastic properties are characterized by a \textit{stored energy function} $\, E : \mathbb C \times \mathbb C \times \mathbb R^{2\times 2} \, \rightarrow \mathbb R\,$,  and $\, f : \mathbb C \onto \mathbb C\,$ as a deformation of finite energy,
\begin{equation}
\mathbf E[f]  \, \bydef \, \int_\mathbb C  E(z, f, Df)\,\mathbf d z  < \infty\,.
\end{equation}
Hereafter the differential matrix $\,Df(z) \in \mathbb R^{2 \times 2}\,$ is referred to as \textit{deformation gradient}. A Sobolev homeomorphism $\, f : \mathbb C \onto \mathbb C\,$  of finite energy is understood as a \textit{hyper-elastic deformation} of $\,\mathbb C\,$. Our  concept of finite energy,  suited to the purpose of the present paper,  is clearly inspired by \textit{mappings of finite distortion}~\cite{AIMb, HK, IMb}, including \textit{quasiconformal mappings}. Therefore, omitting necessary details, the stored energy function will take the form $\,E(z, f, Df) = E(z, |Df|^2/ \textnormal{det} Df)\,$.  We adopt interpretations from NE where a great part of our paper is highly motivated. Let us take a quick look at such mappings.

\subsection{Mappings of finite distortion} Throughout this paper the domain of definition of such mappings consists of sense preserving homeomorphisms $\,f : \mathbb C \onto \mathbb C\,$ of Sobolev class $\,\mathscr W^{1,1}_{\textnormal {loc}} (\mathbb C , \mathbb C)\,$.
\begin{definition} A homeomorphism   $\,f \in \mathscr W^{1,1}_{\textnormal {loc}} (\mathbb C , \mathbb C)\,$  is said to have {\it finite distortion} if there is a measurable function  $\,K : \mathbb C \rightarrow [1, \infty)\,$ such that
\begin{equation}\label{eq:dist}
|Df(z)|^2 \le K(z) {J_f(z)}\;,\;\;\; \textnormal{for almost every\,} \,z \in \mathbb C .
\end{equation}
\end{definition}

Hereafter $\,|Df(z)|\,$ stands for the operator norm of the differential matrix  $\,Df(z) \in \mathbb R^{2 \times 2}\,$,  and ${J_f(z)}$ for its determinant. The smallest function $K(x) \geqslant 1$ for which~\eqref{eq:dist} holds is called the {\it  distortion} of $f$, denoted by $K_f=K_f(x)$.    In terms of d'Alembert complex derivatives, we have  $\,|Df(z)| = |f_z| + |f_{\bar z}|\,$ and $\,J_f(z) = |f_z|^2  - |f_{\bar{z} }|^2\,$.   Thus $\,f\,$ can be viewed as  a \textit{very weak solution} to the \textit{Beltrami equation}:
\begin{equation}\label{BeltramiEquation}
\frac{\partial f}{\partial \bar z} =  \; \mu (z)\, \frac{\partial f}{\partial z}\,\,,  \;\;\;\textnormal{where}\;  \abs{\mu(z)} = \frac{K_f(z) -1}{K_f(z) +1} \, < 1
\end{equation}

\begin{figure}[htbp]
\centering
\includegraphics[width=1.00\textwidth]
{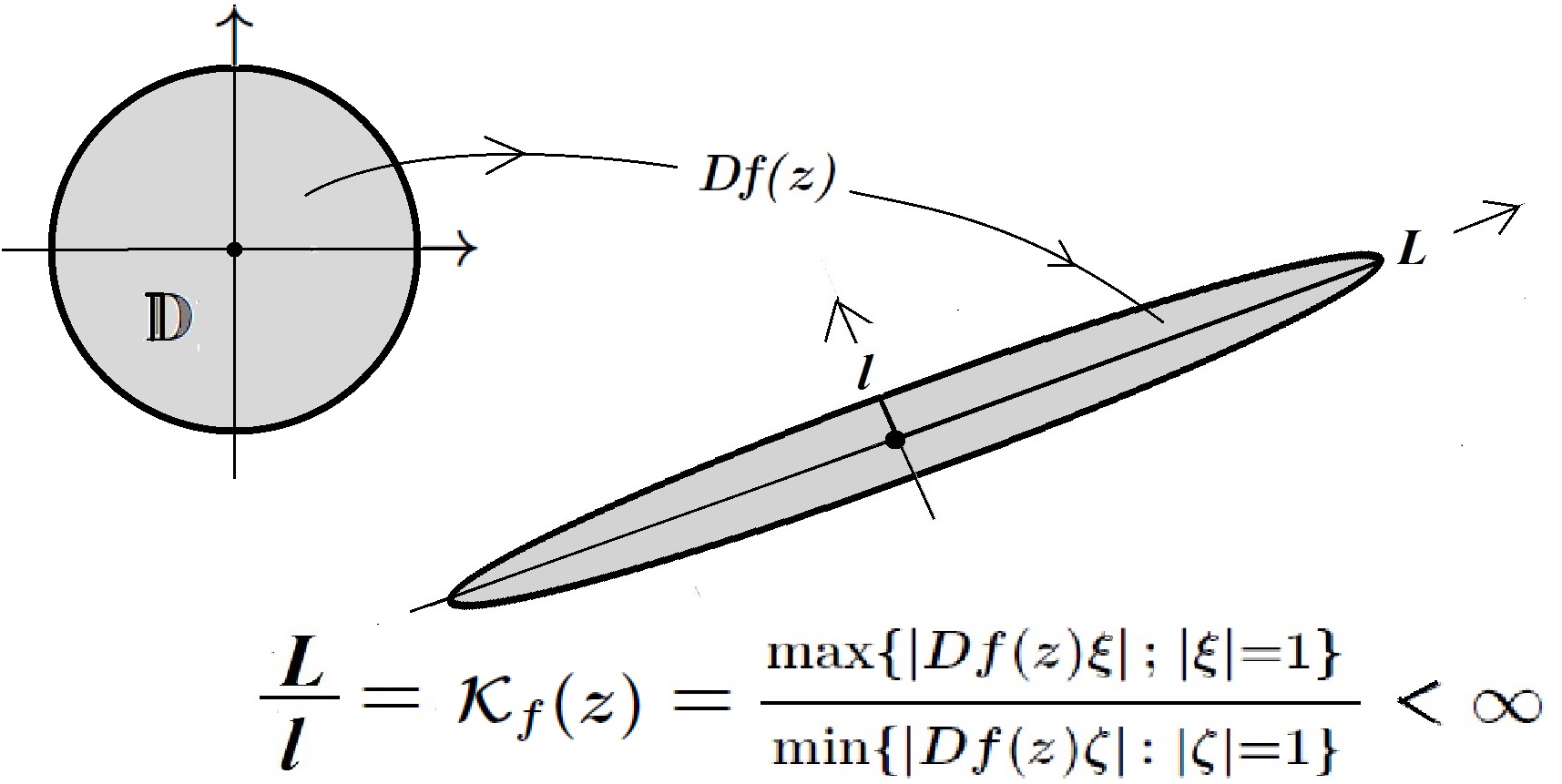}
\caption{The ratio $\,L/l\,$, which measures the infinitesimal distortion of the material structure at the point $\,z\,$,  is allowed to be arbitrarily large. Nevertheless, $\,L/l\,$  has to be finite almost everywhere.}
\label{fig:distortion}
\end{figure}

The distortion inequality~\eqref{eq:dist}  asks that  $\,Df(z) = 0 \in \mathbb R^{2\times 2}\,$ at the points  where  the Jacobian $\, J_f(z)  = \det Df(z)\,$ vanishes.
\begin{definition} A homeomorphism $\,f : \mathbb C \onto \mathbb C\,$ of Sobolev class $\,\mathscr W^{1,1}_{\textnormal {loc}} (\mathbb C , \mathbb C)\,$  is said to be {\it quasiconformal} if $\,K_f \in {\mathscr L}^\infty({\mathbb C})\,$. It is $\,K\,$-\textit{quasiconformal} $( 1 \leqslant K < \infty)\,$ if $\,1\leqslant \, K_f(z) \leqslant K\,$ everywhere.
\end{definition}
\subsection{Quasi-equivalence}
 It should be pointed out that the inverse map $\,f^{-1} : \mathbb C \onto \mathbb C\,$ is also $\,K\,$-quasiconformal and a composition $\,f \circ g\,$ of $\,K_1\,$ and $\,K_2\,$-quasiconformal mappings is $\,K_1\cdot K_2\,$-quasiconformal. \\
 These special features of quasiconformal mappings furnish an equivalence relation between subsets of  $\, \mathbb C\,$ that is reflexive, symmetric and transitive.
 \begin{definition} We say that $\,\mathbb X \subset \mathbb C\,$ is \textit{quasi-equivalent} to  $\,\mathbb Y \subset \mathbb C\,$, and write $\,\mathbb X\quasi \mathbb Y\,$,  if $\, Y = f(\mathbb X) \,$ for some quasiconformal mapping  $\,f : \mathbb C \onto \mathbb C\,$.
 \end{definition}

\subsection{Quasidisks}  One exclusive class of quasi-equivalent subsets is represented by the open unit disk $\,\mathbb D \subset \mathbb C\,$. Thus we introduce the following:
\begin{definition}
A domain  $\,\mathbb X \subset \mathbb C\,$ is called \textit{quasidisk} if it admits a quasiconformal mapping $\, f :\mathbb C \onto \mathbb C\,$  which takes $\,\mathbb X\,$ onto  $\,\mathbb D\,$. In symbols, we have $\,\mathbb X\quasi \mathbb D\,$.
\end{definition}

Quasidisks have been studied intensively for many years because of their exceptional functional theoretical properties, relationships with Teichm\"uller theory and Kleinian groups and interesting applications in complex dynamics, see~\cite{GeMo} for an elegant  survey. Perhaps the best know geometric characterization for a quasidisk is the  {\it Ahlfors' condition}~\cite{Ahre}.
\begin{theorem}[Ahlfors] Let $\X$ be a  (simply connected) Jordan domain in the plane. Then $\X$ is a quasidisk if and only if  there is a constant $\,1 \leqslant \gamma\ <\infty ,$ such that for each pair of distinct points $a, b \in \partial \X$ we have
\begin{equation}\label{eq:ah} \diam \Gamma   \le \gamma \, \abs{a-b}\end{equation}
where $\,\Gamma\,$ is the component of $\partial \X \setminus \{a, b\}\,$ with smallest diameter.
\end{theorem}
\newpage

\begin{figure}[htbp]
\centering
\includegraphics[width=1.00\textwidth]
{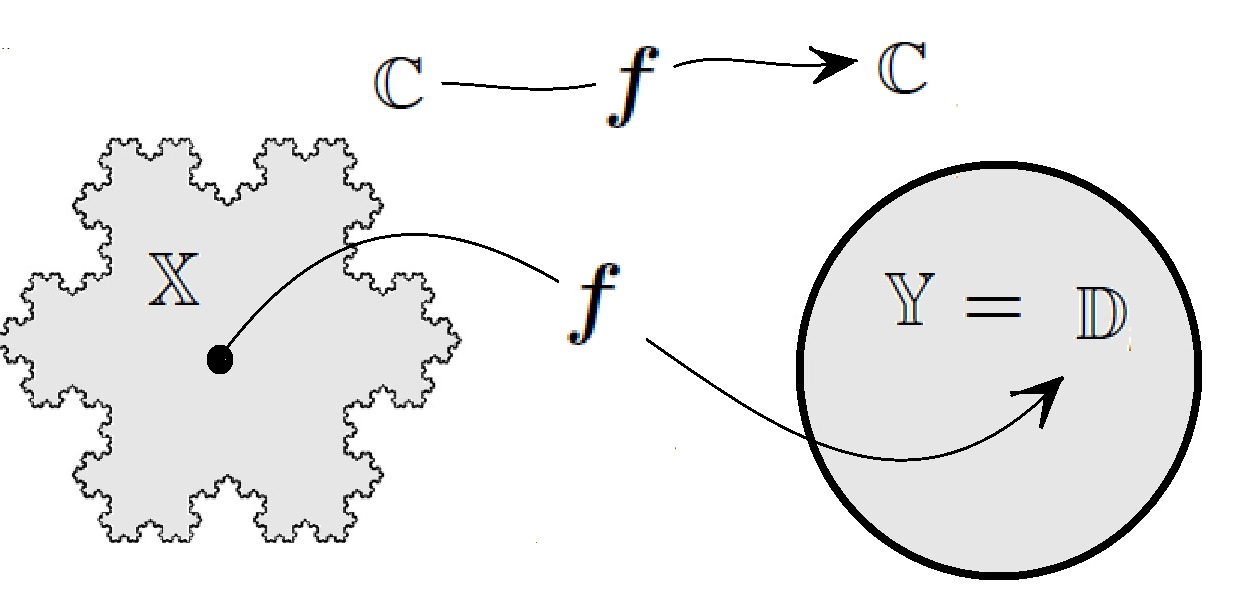}
\caption{ Koch snowflake reveals complexity of a quasidisk.}
\label{fig:snowflake}
\end{figure}
One should infer from the Ahlfors' condition~\eqref{eq:ah} that: \\
\begin{center}

\framebox{\textbf{Quasidisks do not allow for cusps in the boundary.}}

\end{center}

$\;$\\
That is to say, unfortunately, the point-wise inequality  $\,K_f(z) \leqslant K < \infty\,$ precludes $\,f\,$ from smoothing even basic singularities. It is therefore of interest to look for more general deformations $\,f : \mathbb C \onto \mathbb C\,$. We shall see, and it will become intuitively clear, that the act of deviating from conformality should be measured by integral-mean distortions rather than point-wise distortions.
More general  class of mappings, for which one might hope to build  a viable theory, consists of homeomorphisms  with locally $\,\mathscr L^{\,p}\,$-integrable distortion, $\,1 \leqslant p < \infty\,$.
\begin{definition}
The term {\it mapping of $\,\mathscr L^{\,p}\,$-distortion}, $\,1 \le p < \infty$, refers to a  homeomorphism $f \colon \mathbb C \to \mathbb C$ of class $\W_{\loc}^{1,1} (\C, \C)$ with  $K_f \in \mathscr L_{\loc}^{\, p} (\C)$. 
\end{definition}

Now, we generalize the notion of quasidisks; simply, replacing the assumption  $\,K_f \in {\mathscr L}^\infty({\mathbb C})\,$ by  $\,K_f \in {\mathscr L}_{\loc}^{\, p}({\mathbb C})\,$.
\begin{definition} A domain $\,\mathbb X\subset \mathbb C\,$ is called  an {\it $\mathscr L^{\, p}$-quasidisk} if it admits a homeomorphism  $f \colon \C \to \C$ of  $\,\mathscr L^{\,p}\,$-distortion such that $\,f(\mathbb X) = \mathbb D\,$.
\end{definition}
Clearly, $\mathscr L^p$-quasidisks are Jordan domains.
Surprisingly, the $\mathscr L_{\loc}^1$-integrability of the distortion  seems not to cause any geometric constraint on $\,\mathbb X\,$. We confirm this observation for domains with rectifiable boundary.
\begin{theorem}\label{theorem2}
  Simply-connected Jordan domains  with rectifiable boundary are $\cl^{\, 1}$-quasidisks.
\end{theorem}

Nevertheless, the $\mathscr L^{\, p}$-quasidisks with $\,p>1\,$  can be characterized by  model singularities at their boundaries.
The most specific singularities, which fail to  satisfy the Ahlfors' condition~\eqref{eq:ah}, are cusps.  Let
us consider the power-type  inward and outward cusp domains, see Figure~\ref{fig:cusps}.  For $\beta>1$ we consider a disk with inward cusp defined  by
\[\mathbb D^\prec_\beta =\mathbb B(1-\beta, r_\beta) \setminus \{ z=x+iy \in \mathbb C \colon x \ge 0, \abs{y} \le x^\beta \}\;,\;\;r_\beta = \sqrt{\beta^2+1} \, . \]
Whereas a disk with outer cusp will be defined  by
\[\mathbb D^\succ_\beta = \{z=x+iy \in \C \colon 0<x<1, \abs{y} < x^\beta\} \cup \mathbb B(1+\beta, r_\beta) \, . \]
Here, $r_\beta = \sqrt{\beta^2+1}$.

\begin{figure}[htbp]
\centering
\includegraphics[width=1.00\textwidth]
{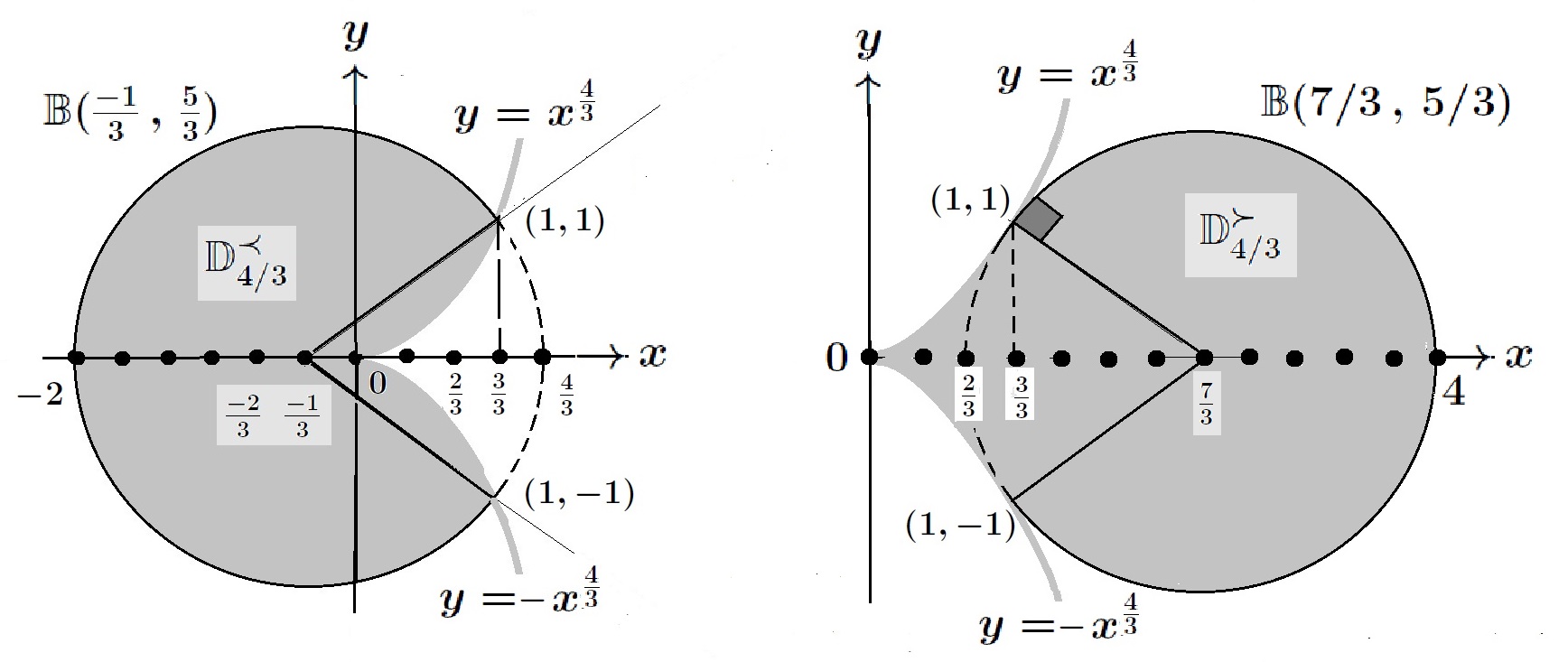}
\caption{The inner and outer \textit{power cusps} in the disks $\,\mathbb D^\prec_\beta\,$  and $\,\mathbb D^\succ_\beta\,$\;, with $\,\beta = \frac{4}{3}\,.\,$}
\label{fig:cusps}
\end{figure}

Note,  all of these domains fail to satisfy the Alhfors' condition~\eqref{eq:ah}. However, replacing $|a-b|$ in~\eqref{eq:ah} by $|a-b|^\alpha$ we obtain:
\begin{definition}
A Jordan domain $\X \subset \C$ is  {\it $\alpha$-Ahlfors regular}, with $\,\alpha \in (0,1]\,$,  if there is a constant $\,1 \leqslant \gamma\ <\infty \,$ such that for each pair of distinct points $a, b \in \partial \X$ we have
\begin{equation}\label{eq:ahalpha} \diam \Gamma   \le \gamma \, \abs{a-b}^\alpha\end{equation}
where $\,\Gamma\,$ is the component of $\partial \X \setminus \{a, b\}\,$ with smallest diameter.
\end{definition}

\begin{theorem}\label{thm:lp}
Let $\X$ be either $\mathbb D^\prec_\beta$ or $\mathbb D^\succ_\beta$ and $1<p<\infty$. Then $\X$ is a $\mathscr L^{\, p}$-quasidisk if and only if $\beta < \frac{p+3}{p-1}$; equivalently, $p < \frac{\beta+3}{\beta-1}$.
\end{theorem}
This simply means that $\,\X\,$ is $\,\frac{1}{\beta}\,$-Ahlfors regular. Theorem~\ref{thm:lp} tells us how much  the distortion of  a homeomorphism $f \colon \C \to \C$ is needed to flatten (or smoothen) the  power type cusp $t^\beta$. It turns out that  a lot more distortion  is needed  to create a cusp than to smooth it back.  Indeed, in a series of papers~\cite{KT1, KT3, KT2},  Koskela and Takkinen  raised such an inverse  question.  For which cusps does there exist a homeomorphism $\,h \colon \mathbb C \to \mathbb C\,$ of finite distortion $\, 1 \leqslant K_h <\infty\,$ which takes $\mathbb D$ onto  $\mathbb D^\prec_\beta\,$?  A necessary condition turns out to be that
$e^{K_h} \not \in \mathscr L_{\loc}^{\, p} (\C)$ with $\,p> \frac{2}{\beta-1}$. However, if $p< \frac{2}{\beta-1}$ there is such a homeomorphism. Especially, each power-type  cusp domain can be obtained as the image of open disk by a homeomorphism $h \colon \C \to \C$ with $K_h \in \mathscr L_{\loc}^{\, p} (\C)$ for all $p<\infty$. Combining this with Theorem~\ref{thm:lp} boils down to the following postulate:
\begin{center}
{\it Creating singularities takes almost no efforts (just allow for a little distortion) while tidying them up is a whole new story.}
\end{center}

\subsection{The energy  for $\,\mathscr L^{\,p}\,$-distortion}
 We need to pullback to $\,\mathbb C\,$ the Euclidean area element $\,\textnormal{d}\sigma(\xi) \,$ of $\,\mathbb S^2 \subset \mathbb R^3\,$ by stereographic projection $\,\Pi : \mathbb S^{\circ}  \onto \mathbb C\,$, where
$$ \mathbb S^{\circ} \bydef \{\,\xi = (w,t) \colon w \in \mathbb C\,,\, -1 \leqslant t < 1\,,\, |w|^2  + t^2 = 1 \,\} \subset \mathbb C \times \mathbb R  \cong \mathbb R^3 \, . $$
The image point $\,z = \Pi\xi\,$ is defined by the rule $\,\Pi(w,t) = \frac{w}{1-t}\,$. For the inverse projection $\,\Pi^{-1}\,: \,\mathbb C \onto \mathbb S^{\circ}\,$ we have:
$$
\xi = \Pi^{-1}z = (w,t)\;,\; \textnormal{ where } \; w = \frac{2\,z}{1+|z|^2} \;\; \textnormal{and}\; \; t = \frac{|z|^2 -1}{|z|^2 + 1} \, . 
$$
Denote by $\,\mathbf dz\,=\, \textnormal{d}x\,\textnormal{d} y\,$ the area element in $\,\mathbb C\,$,  $\, z = x+ iy \,$. The  general formula of integration by change of variables  reads as follows:
$$
\textnormal{d}\sigma(\xi) =  \frac{4\,\mathbf dz }{( |z|^2 + 1) ^2}\;,\; \textnormal{hence}\;  \;\int_\mathbb C  \frac{4\, G(z)\,\mathbf dz}{(|z|^2 + 1 )^2} = \int_{\mathbb S^{\circ}} G(\Pi\xi) \,\textnormal{d}\sigma(\xi)
$$
Now, one might consider mappings of $\,\mathscr L^{\,p}\,$-distortion which have finite $\,\mathscr L^{\,p}\,$-energy:
\begin{equation} \label{Energy}
\mathbf E[f] \bydef 4 \int_\mathbb C \frac{ [ K_f(z) ]^p\,\mathbf dz}{(|z|^2 + 1 )^2} = \int_{\mathbb S^{\circ}}[  K_{\mathcal F}(\xi) ]^p \,\textnormal{d}\sigma(\xi) < \infty ,
\end{equation}
 where $\,K_\mathcal F : \mathbb S^{\circ} \,\rightarrow [1, \infty)\,$ stands for the distortion function of the mapping $\,\mathcal F =  f \circ \Pi  : \mathbb S^{\circ}  \onto \mathbb C\,$. For the energy formula~\eqref{Energy}, we invoke the equality $\, K_{\mathcal F}(\xi) = K_f(\Pi\xi)\,$ which is    due to the fact that $\,\Pi\,$ is conformal. This formula makes it clear that $\,K\,$-quasiconformal mappings $\,f : \mathbb C \onto \mathbb C\,$ have finite $\,\mathscr L^{\,p}\,$-energy and $\,\mathbf E[f] \leqslant 4 \pi  K^p\,$. 
 
 In the spirit of extremal quasiconformal mappings in Teichm\"uller spaces, one  might be interested in studying homeomorphisms $\,f : \mathbb C \onto \mathbb C\,$  of smallest $\,\mathscr L^{\,p}\,$-energy, subject to the condition $\,f(\mathbb X)  = \mathbb Y\,$. Here the given pair $\,\mathbb X , \mathbb Y\,$ of subsets in $\,\mathbb C\,$ is assumed to admit at least one such homeomorphism of finite energy. To look at a more specific situation, take for $\,\mathbb X\,$ an $\,\mathscr L^{\,p}\,$-quasidisk  from  Theorem \ref{thm:lp}, and the unit disk $\,\mathbb D\,$  for $\,\mathbb Y\,$. What is then the energy-minimal map $\,f : \mathbb C \onto \mathbb C\,$? Polyconvexity of the integrand will certainly help us find what conditions are needed for the existence of energy-minimal mappings. 
 We shall not enter these topics here, but refer to \cite{AIMO, IOlp, MMc}   for related results.
 
\subsection{The main result}
Since a simply connected Jordan domain is conformally equivalent with the unit disk, it is natural to consider special  $\mathscr L^p$-quasidisks; namely, the domains $\X$ which can be mapped onto an open disk under a homeomorphism $f \colon \C \to \C$ with $p$-integrable distortion and   to be  quasiconformal when restricted to $\,\X\,$.

The answer to this question can be inferred from  our main result which also generalizes Theorem~~\ref{thm:lp}.
\begin{theorem}[Main Theorem] \label{thm:main}
Consider power-type inward cusp domains $\X=\mathbb D^\prec_\beta\,$ with $\beta>1\,$.  Given a pair $\,(q,p)\,$ of exponents  $\, 1\leqslant  q  \le \infty\,$ (for $\mathbb X\,$) and $\,1< p \leqslant \infty\,$ (for the complement of $\,\mathbb X\,$), define the so-called critical power of inward cusps
\begin{equation}\label{eq:betac} \beta_{\,\textnormal{cr}}  \bydef \begin{cases} \frac{p\,q\,+\,2\,p\,+\,q}{p\,q\,-\,q} \, , \quad & \textnormal{ if } 1<p<\infty \textnormal{ and } q< \infty \\
\frac{2}{q}+1\, ,  & \textnormal{ if } p=\infty  \,\;\textnormal {and}\;\; q < \infty\\
\frac{p\,+\,1}{p\,-\,1}\, ,  & \textnormal{ if } 1< p < \infty \;\;\textnormal {and} \; q=\infty
\end{cases}\end{equation}
Then there exists a Sobolev homeomorphism $f \colon \C \to \C$ which takes $\X$ onto $\mathbb D$ such that
\begin{itemize}
\item $K_f \in \mathscr L^{\, q} (\X) $
\end{itemize}
and
\begin{itemize}
\item $K_f \in  \mathscr L^{\, p} (\mathbb B_R \setminus \overline{ \X })\,$ for every $R>2$,
\end{itemize}
  if and only if $\beta < \beta_{\,\textnormal{cr}}$.
\end{theorem}
Here and what follows $\mathbb B_R= \{z \in \mathbb C \colon \abs{z} <R\}$ for $R>0$. 
\begin{figure}[htbp]
\centering
\includegraphics[width=1.00\textwidth]
{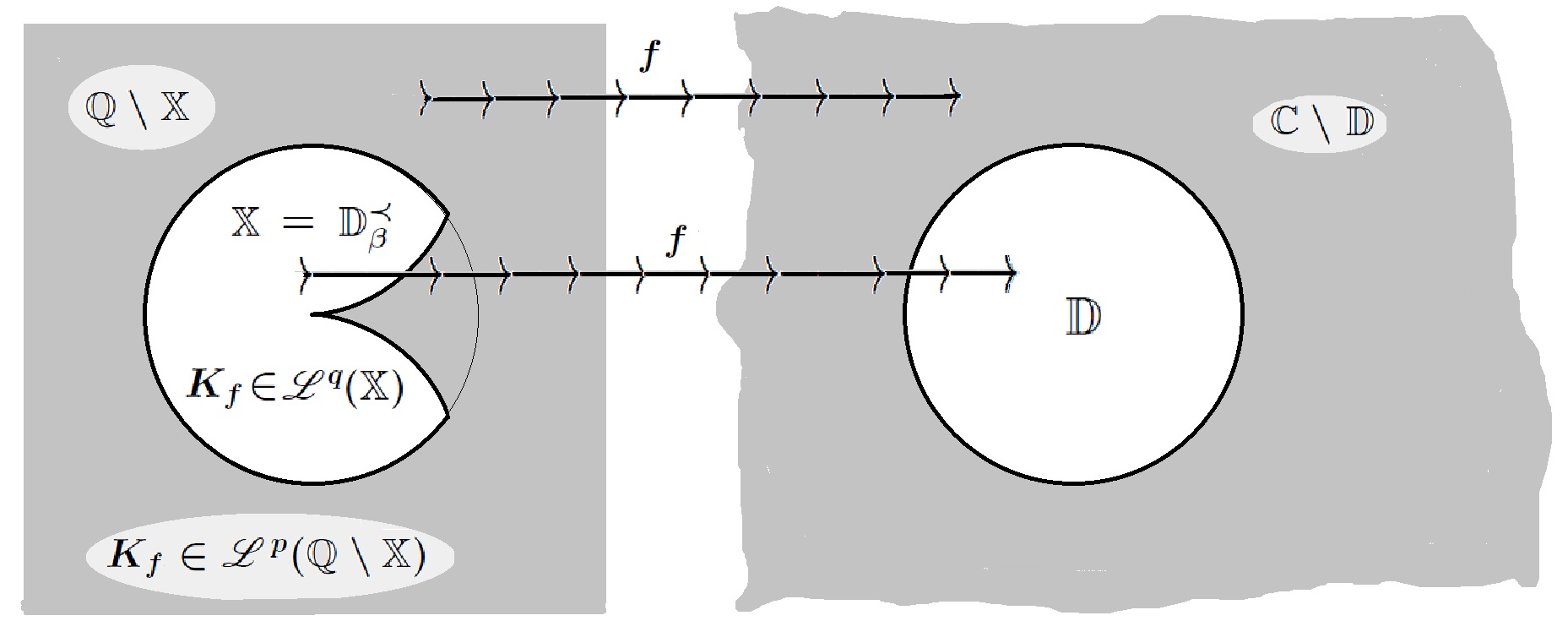}
\caption{An $\,\mathscr L^{q, p}\,$-quasidisk}
\label{fig:qpquasidisk}
\end{figure}
Applying  the standard inversion of  unit disk, Theorem~\ref{thm:main} extends to the power-type outer cusp domains as well. In this case the roles of $p$ and $q$ are interchanged. The reader interested in learning more about the conformal case  $f \colon \mathbb D^\prec_\beta  \onto \mathbb D$ is refer to~\cite{Xu}.

Our proof of Theorem~\ref{thm:main}  is self-contained. The ``only if''  part of Theorem~\ref{thm:main}   relies on  a regularity estimate  of a reflection in $\partial \mathbb D^\prec_\beta$.  Such a reflection is defined and examined in the boundary of an arbitrary $\mathscr L^p$-quasidisk.  In this connection  we recall  a classical result of K\"uhnau~\cite{Ku} which  tells us that a Jordan domain  is a qusidisk if and only if it
admits a qusiconformal reflection in its boundary. Before going into details about the boundary reflection proceeders  (Section \ref{sec:ref}) we need some preliminaries.\\

\section{Preliminaries}

First we recall a well-known theorem of Gehring and Lehto~\cite{GL} which asserts that a planar  open mapping with finite partial derivatives at almost every point is differentiable at almost every point. For homeomorphisms the result was earlier established by Menchoff~\cite{Me}.

\begin{lemma}\label{lem:differen}
Suppose that $f \colon \C\rightarrow\C$ is a homeomorphism in the class $\W^{1,1}_{\rm loc}(\C,\C)$. Then $f$ is differentiable almost everywhere.
\end{lemma}

It is easy to see, at least formally, applying a change of variables  that the integral of distortion function  equals the Dirichlet integral of  inverse mapping.
This observation is the key to the fundamental identity which we state next, see~\cite{HK, HKO, MP}.
\begin{lemma}\label{lem:intdist}
Suppose that a homeomorphism $f \colon \C \onto \C$ of Sobolev class $\W^{1,1}_{\loc} (\C, \C)$. Then $f$ is a mapping of $\mathscr L^{1}\,$-distortion if and only if  the inverse $h\bydef f^{-1}\in W^{1,2}_{\loc} (\C, \C)$. Furthermore, then for every  bounded domain $\U \subset \C$ we have
\[ \int_{f(\U)} \abs{Dh(y)}^2 \, \dtext y = \int_\U K_f(x) \, \dtext x \]
and  $J_f(x)>0$ a.e.
\end{lemma}

At least formally  the identity $(h \circ f) (x)=x$, after differentiation,  implies that $Dh (f(x)) Df(x)= {\bf I}$. The validity of such identity under minimal regularity assumptions on the mappings is the essence of the following lemma, see~\cite[Lemma A.29]{HK}.
\begin{lemma}\label{lem:inverse}
Let $f\colon \X\to\Y$ be a homeomorphism which is differentiable at $x\in\X$ with $J_f(x)>0$. Let $h\colon \Y\to\X$ be the inverse of $f$. Then $h$ is differentiable at $f(x)$ and $Dh(f(x))=(Df(x))^{-1}$.
\end{lemma}

Next we state a crucial version of the area formula for us.

\begin{lemma}\label{lem:co-area}
Let $\X, \Y\subset \C$ be domains and  $g \colon \X \onto \Y$ a homeomorphism.  Suppose that $\V \subset \X$ be a measurable set and $g$ is differentiable at every point of $\V$. If  $\eta$ is a nonnegative Borel measurable function, then
\begin{equation}\label{eq:cvf}
\int_{\V}\eta(g(x))|J_g(x)|\, \dtext x\le \int_{g(\V)}\eta(y)\, \dtext y\, .
\end{equation}
\end{lemma}
This follows from~\cite[Theorem 3.1.8]{Federer} together with the area formula for Lipschitz mappings.

The circle is uniquely characterized by the property that among all  closed Jordan curves of given length $L$, the circle of circumference $L$  encloses maximum area. This property is  expressed in the well-known  isoperimetric inequality.
\begin{lemma}\label{lem:isoperimetric}
Suppose $\mathbb U$ is a bounded Jordan domain with rectifiable boundary $\partial \mathbb U$. Then
\begin{equation}
\abs{\mathbb U} \le \frac{1}{4 \pi} [\ell (\partial \mathbb U)]^2
\end{equation}
where $\abs{\mathbb U}$ is the area of $\mathbb U$ and $\ell (\partial \mathbb U)$ is the length of $\partial \mathbb U$.
\end{lemma}

\section{Reflection}\label{sec:ref}
We denote the one point compactification of the complex plane by $\widehat\C \bydef \C\cup\{\fz\}$.
\begin{definition}
A domain $\Omega \subset \widehat\C$ admits a {\it reflection in its boundary} $\partial \Omega$ if  there exists a
homeomorphism $g$ of $\widehat\C$ such that
\begin{itemize}
\item $g(\Omega) = \widehat\C \setminus \overline{\Omega}$, and
\item $g(z)=z$ for $z\in \partial \Omega$.
\end{itemize}
\end{definition}
A domain $\Omega \subset  \widehat\C$ is a Jordan domain if and only if it admits a reflection in its boundary, see~\cite{GeHag1}. In this section we raise a question what else can we say about the reflection if the domain is an $\mathscr L^p$-quasidisk.  A classical result of K\"uhnau~\cite{Ku} tells us that $\Omega \subset  \widehat\C$  is a qusidisk if and only if it
admits a qusiconformal reflection in $\partial \Omega$.
  Let $\X \subset \C$ be an $\mathscr L^p$-quasidisk. Then there exists a homeomorphism $f \colon \C \onto \C$ such that $f(\X)= \mathbb D$. We extend $f$  by setting $f(\infty)=\infty$ and still denote the extended mapping by $f$. This way we obtain a homeomorphism $f\colon \widehat {\C} \onto \widehat{\C}$.
We also denote its inverse  by $h \colon \widehat {\C} \onto \widehat{\C}$.

  The circle inversion map $\Psi \colon \widehat {\C} \onto \widehat{\C}$,
  \[\Psi (z) \bydef \begin{cases} \frac{z}{\abs{z}^2} \quad & \textnormal{ if } z \not=0  \\
  \infty & \textnormal{ if } z =0 \end{cases} \]
   is anticonformal, which means that at every point it preserves angles and reverses orientation. The circle inversion  defines  a reflection in $\partial \X$ by the rule
\begin{equation}\label{eq:reflection}
g \colon \widehat {\C} \onto \widehat{\C} \qquad g(x) \bydef h \circ \Psi \circ f (x) \, .
\end{equation}
\begin{theorem}\label{thm:reflection}
Let $\X$ be an $\mathscr L^p$-quasidisk  and $g$ the reflection in $\partial \X$ given by~\eqref{eq:reflection}. Then for a bounded domain $\U \subset \C$ such that $h(0)\not \in \overline{\U}$ we have $g \in \W^{1,1} (\U, \C)$ and
\begin{equation}\label{doudis}
\int_\U\frac{|Dg(x)|^{p}}{|J_g(x)|^{\frac{p-1}{2}}}\, \dtext x\le \lf(\int_{g(\U)}K^p_f(x)\, \dtext x\r)^{\frac{1}{2}}\cdot\lf(\int_\U K^p_f(x)\, \dtext x\r)^{\frac{1}{2}} \, .
\end{equation}
\end{theorem}
\begin{proof}
Let $\U$ be a bounded domain in $\C$ such that $h(0) \not \in \overline{ \U}$. For $x\in \U$ we denote
\[\tilde{f} (x) \bydef \Psi \circ f (x)  \quad \textnormal{and} \quad  \tilde{h} (y) \bydef (\tilde{f})^{-1} (y) \, . \]
We write
\[\V \bydef \{x\in \U \colon f \textnormal{ is differentiable at } x \; \textnormal{ and } \;  J_f(x)>0  \} \, . \]
Then by Lemma~\ref{lem:differen} and Lemma~\ref{lem:intdist} we obtain $\abs{\V}= \abs{\U}$.

Fix $x\in \V$. Then $\tilde{f}$ is differentiable at $x$. Furthermore,  $h$ is differentiable at $f(x)$, see Lemma~\ref{lem:inverse}. Therefore, for $x\in \V$ the chain rule gives
\begin{equation}
\abs{Dg(x )}\le \abs{Dh(\tilde f(x))}  \, \abs{D\tilde f(x)}\ \textnormal{ and }   J_g(x)=J_h(\tilde f(x)) J_{\tilde f}(x)\, .
\end{equation}
Hence, applying H\"older's inequality we have
\begin{equation}\label{eq:main78}
\begin{split}
\int_\U \frac{\abs{Dg(x)}^p}{\abs{J_g (x)}^\frac{p-1}{2}} \, \dtext x & =  \int_\V \frac{\abs{Dg(x)}^p}{\abs{J_g (x)}^\frac{p-1}{2}} \, \dtext x \\
&\le \int_\V \frac{|Dh(\tilde f(x))|^{p}}{|J_h(\tilde f(x))|^{\frac{p-1}{2}}}\frac{|D\tilde f(x)|^{p}}{|J_{\tilde f}(x)|^{\frac{p-1}{2}}}\,  \dtext x \\
&\le\lf(\int_\V \frac{|Dh(\tilde f(x))|^{2p}}{|J_h(\tilde f(x))|^{p-1}}|J_{\tilde f}(x)|\, \dtext x\r)^{\frac{1}{2}}\\
                                                                  & \; \;  \cdot\lf(\int_\V \frac{|D\tilde f(x)|^{2p}}{|J_{\tilde f}(x)|^p}\, \dtext x\r)^{\frac{1}{2}}.
\end{split}
\end{equation}
According to Lemma~\ref{lem:co-area} we obtain
\begin{equation}\label{eq:78}\int_\V \frac{|Dh(\tilde f(x))|^{2p}}{|J_h(\tilde f(x))|^{p-1}}|J_{\tilde f}(x)| \dtext x \le \int_{\tilde f(\V)} \frac{\abs{Dh(y)}^{2p}}{[J_h(y)]^{p-1}} \, \dtext y  \, . \end{equation}
Applying Lemma~\ref{lem:co-area} again this time for $h$, we have
\[  \int_{\tilde f(\V)} \frac{\abs{Dh(y)}^{2p}}{[J_h(y)]^{p}} J_h(y) \, \dtext y \le \int_{g(\V)} [Dh \big( f(x)\big )]^{2p}  [J_f (x)]^p \, \dtext x \, . \]
This together with Lemma~\ref{lem:inverse} gives
\[  \int_{\tilde f(\V)} \frac{\abs{Dh(y)}^{2p}}{[J_h(y)]^{p}} J_h(y) \, \dtext y \le   \int_{g(\V)} \left[(Df(x))^{-1}\right]^{2p}  [J_f (x)]^p \, \dtext x \, . \]
The familiar Cramer's rule implies
\begin{equation}\label{eq:785}
\int_{g(\V)} \left[(Df(x))^{-1}\right]^{2p}  [J_f (x)]^p \, \dtext x =  \int_{g(\V)} \frac{\abs{Df(x)}^{2p}}{[J_f (x)]^p} \, .
\end{equation}
Combining the estimate~\eqref{eq:78} with~\eqref{eq:785} we have
\begin{equation}\label{eq:7855}
\int_\V \frac{|Dh(\tilde f(x))|^{2p}}{|J_h(\tilde f(x))|^{p-1}}|J_{\tilde f}(x)| \dtext x \le \int_{g(\U)} K_f^p (x) \, \dtext x \, .
\end{equation}
Estimating the second term on the right hand side of~\eqref{eq:main78} we simply note that $\abs{D \Psi (z)}^2 = J(z, \Psi)$ for $z \in \C \setminus \{0\}$ and so
\begin{equation}\label{eq:78551}
\int_\V \frac{|D\tilde f(x)|^{2p}}{|J_{\tilde f}(x)|^p}\, \dtext x = \int_\V K^p_f(x)\, \dtext x \le  \int_{\U} K^p_f(x)\, \dtext x\, .
\end{equation}
The claim follows from~\eqref{eq:main78},~\eqref{eq:7855} and~\eqref{eq:78551}.
\end{proof}

\section{Proof of Theorem \ref{theorem2}}
The proof is based on a  Sobolev variant of the Jordan-Sch\"onflies theorem.
\begin{lemma}\label{lem:extension}
Let $\X$ and $\Y$ be bounded simply connected  Jordan domains, $\D\Y$ being rectifiable. A boundary homeomorphism $\phi \colon \D\X\onto \D\Y$ satisfying \begin{equation}\label{inverse}
\int_{\D\Y}\int_{\D\Y}\lf|\log|\phi^{-1}(\xi)-\phi^{-1}(\eta)|\r| |\dtext \xi| |\dtext \eta|<\fz
\end{equation}
 admits a homeomorphic extension $h\colon \C\to\C$ of Sobolev class $\W^{1,2}_{\loc}(\C,\C)$.
\end{lemma}
This result is from \cite[Theorem 1.6]{KO}. Note that if one asks the existence of homeomorphic extension  $h \colon \overline{\X} \onto \overline{\Y}$ (on one side of $\partial \X$) in the Sobolev class $\W^{1,2} (\X, \C)$. First, applying the Riemann Mapping Theorem we may assume that $\X=\mathbb D$. Second, a necessary condition is that the mapping $\phi$ is the Sobolev trace of some (possibly non-homeomorphic) mapping in $\W^{1,2} (\X, \C)$. The class of boundary functions which admit a harmonic extension with finite Dirichlet energy was characterized by Douglas~\cite{Do}. The Douglas condition for a function  $\phi \colon \D\mathbb D \onto \D\Y$ reads as
\begin{equation}\label{Douglas}
\int_{\D\mathbb D}\int_{\D\dd}\left|\frac{\phi(\xi)-\phi(\eta)}{\xi-\eta}\right|^2|\dtext \xi||\dtext \eta|<\fz.
\end{equation}
 In \cite{AIMO} it was shown that for $\mathscr C^1$-smooth $\Y$ the Douglas condition~\eqref{Douglas} can be equivalently given in terms of the inverse mapping $\phi^{-1} \colon \D\Y\onto \D\dd$ by~\eqref{inverse}. Beyond the $\mathscr C^1$-smooth domains, if $\Y$ is a Lipschitz regular, then a boundary homeomorphism $\phi \colon \D\mathbb D \onto \D\Y$ admits a homeomorphic extension $h \colon \overline{\mathbb D} \onto \overline{\Y}$ in $\W^{1,2}(\mathbb D, \C) $ if and only if $\phi$ satisfies the Douglas condition. There is, however, an inner chordarc domain $\Y$ and a homeomorphism $\phi \colon \partial \mathbb D \onto \partial \Y$
satisfying the Douglas condition which does not admit a homeomorphic extension $h \colon \overline{\mathbb D} \onto \overline{\Y}$ with finite Dirichlet energy.
Recall that $\Y$ is an  inner chordarc domain if there exists a homeomorphism $\Upsilon \colon \overline{\Y} \onto \overline{\mathbb D}$  which is $\mathscr C^1$-diffeomorphic in $\Y$ with bounded gradient matrices $D \Upsilon$ and $(D \Upsilon )^{-1}$. These and more about Sobolev homeomorphic extension results we refer to~\cite{KO}.

\begin{proof}[Proof of Theorem \ref{theorem2}]
Let $\X\subset\C$ be a simply connected Jordan domain, $\D\X$ being rectifiable. According to Lemma~\ref{lem:intdist}, $\X$ is an $\mathscr L^1$-quasidisk if and only if there exists a homeomorphism $h \colon \C \onto \C$ in $\W^{1,2}_{\loc} (\C, \C)$ such that $h(\mathbb D)=\X$. Therefore, by Lemma~\ref{lem:extension} it suffices to construct a boundary homeomorphism $\phi \colon \D\dd\onto \D\X$ which satisfies
\[
\int_{\D\X}\int_{\D\X}\lf|\log|\phi^{-1}(\xi)-\phi^{-1}(\eta)|\r| |\dtext \xi| |\dtext \eta|<\fz \, .
\]

Let $\xi, \eta \in\D\X$ be arbitrary. We denote by  $\gamma_{\xi\eta}$  the subcurve of $\D\X$, connecting    $\xi$  and  $\eta$.   The curve  $\gamma_{\xi\eta}  $ is parametrized counterclockwise. Setting $z_\xi=1$. For arbitrary $z\in\D\dd$ let $\wideparen{z_\xi z}\subset\D\dd$ be the circular arc starting from $z_\xi$ ending at $z$. The arc is parametrized counterclockwise.  For $\eta\in\D\X$, there exists a unique $z_\eta\in\D\dd$ with
\begin{equation}
\frac{\ell (\gamma_{\xi\eta})}{ \ell (\D\X)}=\frac{\ell (\wideparen{z_\xi z_\eta})}{\ell (\D\dd)}.\nonumber
\end{equation}
Now,  we define the boundary homeomorphism $\phi:\D\dd\to\D\X$ by setting $\phi(z_\eta)=\eta$.
\begin{figure}[htbp]
\centering
\includegraphics[width=0.9\textwidth]
{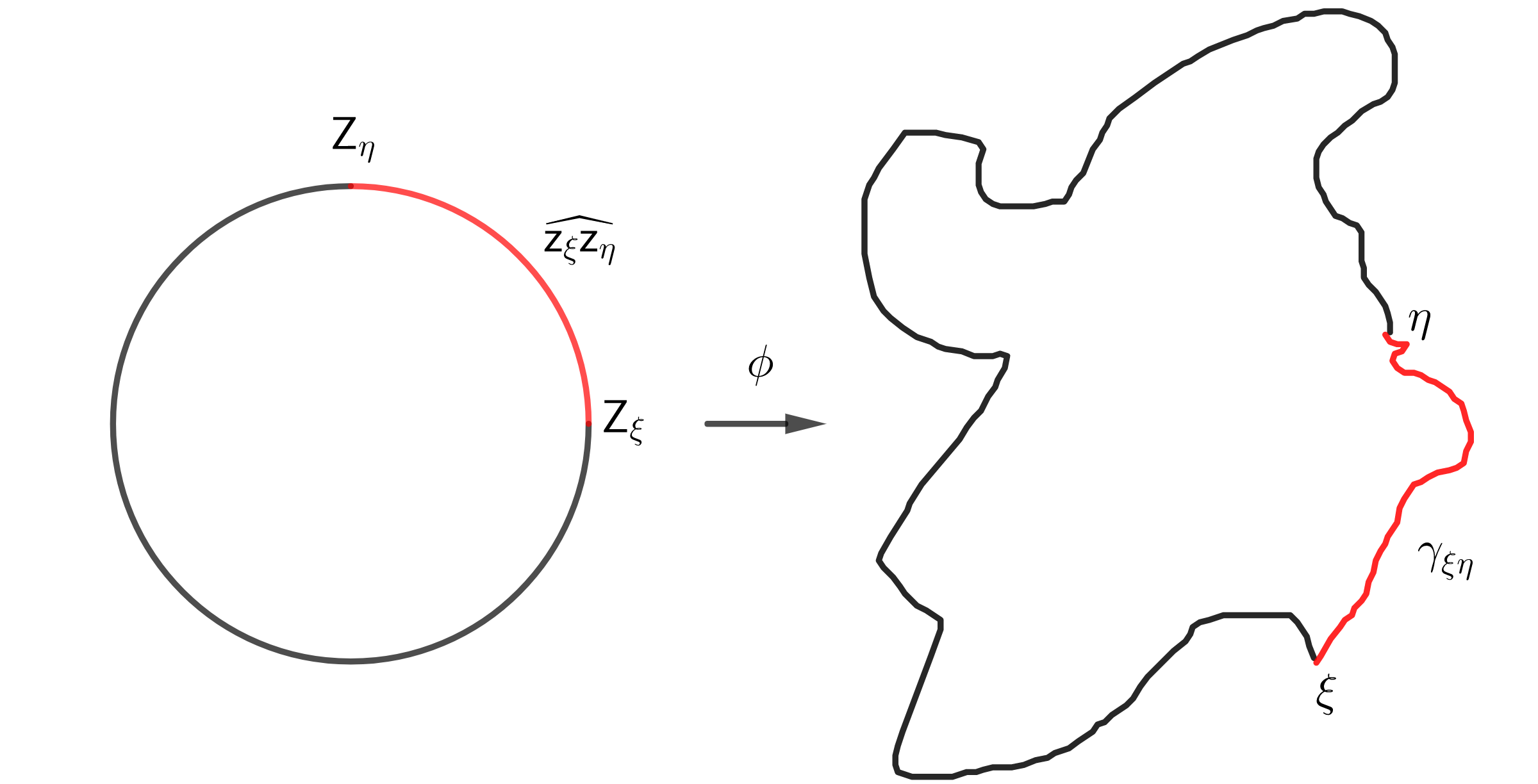}
\end{figure}

First, we observe that  $\abs{\phi'(z)}=\frac{\ell (\D\X)}{\ell (\D\dd)}$ for every $z\in\D\dd$. Furthermore since  the length of the shorter circular arc between two points in $\D\dd$ is comparable to their Euclidean distance   the change of variables formula gives
\begin{eqnarray}
\int_{\D\X}\abs{\log\abs{\phi^{-1}(\xi)-\phi^{-1}(\eta)}}\, \abs{\dtext \eta}&\le&C\int_{\D\dd}\abs{\log\abs{\phi^{-1}(\xi)-\phi^{-1}(\eta)}}\, \abs{\dtext \phi^{-1}(\eta)}\nonumber\\
                                                                                              &\le&C\int_0^{2\pi}\abs{\log t}\, \dtext t<\fz.\nonumber
\end{eqnarray}
\end{proof}

\section{Proof of Theorem \ref{thm:main}}
Before jumping into the proof we fix a few notation and prove two auxiliary results. 
Fix a  power-type inward cusp domain $\dd^\prec_\beta$. For $0<t<1$ we write
\begin{equation}
\mathbb I_t \bydef \{t+iy\in\C \colon 0\le|y|<t^\beta\}\nonumber
\end{equation}
and
\begin{equation}
\U_t \bydef \{x+iy\in\C \colon 0<x<t \textnormal{ and } 0\le|y|<x^\beta\}.\nonumber
\end{equation}
The area of $\U_t$ is given by
\[\abs{\U_t}=\int_0^t\int_{-s^\beta}^{s^\beta}1\, \dtext y \, \dtext s=\frac{2t^{\beta +1}}{\beta+1}\, .\]

Suppose the cusp domain $\dd^\prec_\beta$ is an $\mathscr L^s$-quasidisk for $1\le s <\infty$. Note that according to Theorem~\ref{theorem2} the domain $\dd^\prec_\beta$  is always
 an $\mathscr L^1$-quasidisk for every  $\beta$. Therefore, there exists a homeomorphism $f \colon \C \onto \C$ of $\mathscr L^1$-distortion such that $f(\dd^\prec_\beta)= \mathbb D$. We denote the inverse of $f$ by $h \colon \C \onto \C$. After  first extending the homeomorphisms $f$ and $h$ by $f(\infty)=\infty =h(\infty)$ we define a homeomorphism  $g\colon \widehat{\C} \onto \widehat{\C}$ by the formula~\eqref{eq:reflection}.  The mapping $g$ gives a reflection in the boundary of $\dd^\prec_\beta$; that is,
\begin{itemize}
\item $g (\dd^\prec_\beta) =  \widehat{\C} \setminus \overline{\dd^\prec_\beta}$,
\item $g( \widehat{\C} \setminus \overline{\dd^\prec_\beta}) = \dd^\prec_\beta $ and 
\item $g(x)=x$ for $x\in \partial \dd^\prec_\beta$.
\end{itemize}

 \begin{lemma}\label{lem:epsilon}
Let $\epsilon_n=2^{-n}$ for $n \in \mathbb N$. Then there exists  a subsequence $\{\epsilon_{n_k}\}$ of  $\{\epsilon_n\}$ such that for every $k \in \mathbb N$ we have either
\begin{itemize}
\item    $\abs{g(\U_{\epsilon_{n_k}})} \le \epsilon_{n_k}^2$ or
\item  $\abs{g(\U_{\epsilon_{n_k}})} \le 5 \abs{g(\U_{\epsilon_{n_{k}+1}})} $ and $\abs{g(\U_{\epsilon_{n_k}})} > \epsilon_{n_k}^2$.
\end{itemize}
\end{lemma}
\begin{proof}
Assume to the contrary that  the claim is not true, then there exists $n_o\in\mathbb{N}$ such that for every $i\ge n_o$, we have $|g(U_{\epsilon_i})|>\ez^2_i$ and $|g(U_{\ez_i})|>5|g(U_{\ez_{i+1}})|$. Hence we have
\begin{equation}
|g(U_{\ez_{n_o}})|>5|g(U_{\ez_{n_o+1}})|>...>5^n|g(U_{\ez_{n_0+n}})|>....\nonumber
\end{equation}
which implies that for every $n\in\mathbb N$, we have
\begin{equation}\label{eq:987}
|g(U_{\ez_{n_o}})|>\lf(\frac{5}{4}\r)^n4^{-n_o}.
\end{equation}
Letting $n \to \infty$ the term on the right hand side of~\eqref{eq:987} converges to $\fz$  which contradicts with  $|g(U_{\ez_{n_o}})|<|\dd^\prec_\beta|<\fz$.
\end{proof}

\begin{figure}[htbp]
\centering
\includegraphics[width=0.4\textwidth]
{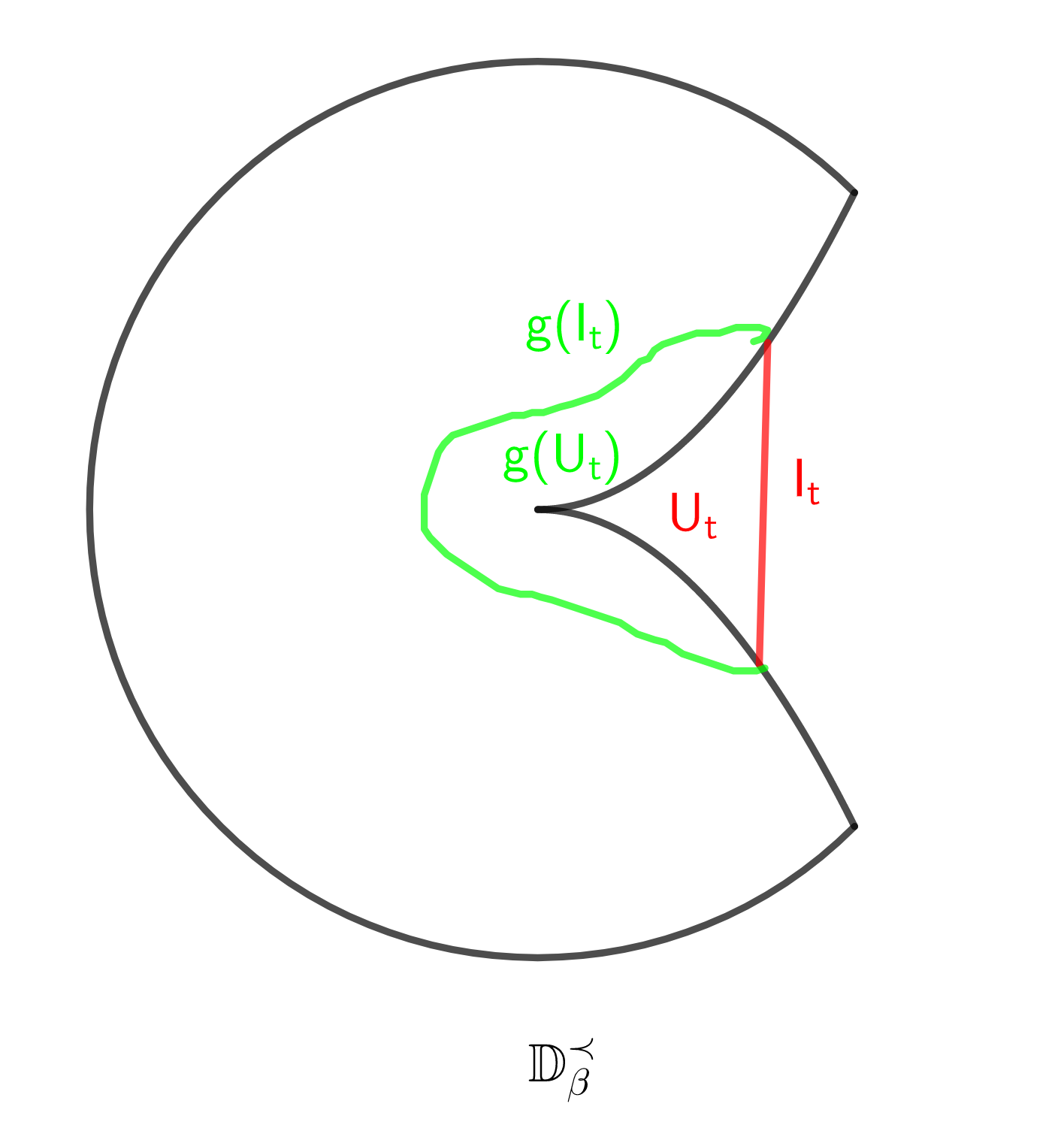}
\end{figure}

The key observation to show that $\dd^\prec_\beta$, $\beta>1$, is not an $\mathscr L^s$-quasidisk for sufficiently large $s>1$ is to compare the length of curves $g(\mathbb I_t)$ and $\mathbb I_t$.

\begin{lemma}\label{lem:key}
Suppose that $\dd^\prec_\beta$ is an $\mathscr L^s$-quasidisk for $1< s<\fz$. Then for almost every $0<t<1$ we have
\begin{equation}\label{eq:ineq1}
\ell (g(\mathbb I_t)) \le \lf(\int_{\mathbb I_t}\frac{|Dg(x)|^s}{|J_g(x)|^{\frac{s-1}{2}}}  \, \dtext x\r)^{\frac{1}{s}}\lf(\int_{\mathbb I_t}|J_g(x)|^{\frac{1}{2}}\, \dtext x\r)^{\frac{s-1}{s}}\, .
\end{equation}
\end{lemma}
\begin{proof}
The second estimate in~\eqref{eq:ineq1} follows immediately from H\"older's inequality
\[
\begin{split}
\ell (g(\mathbb I_t))&\le \int_{\mathbb I_t}|Dg(x)|\, \dtext x\le\int_{\mathbb I_t}\frac{|Dg(x)|}{|J_g(x)|^{\frac{s-1}{2s}}} \cdot |J_g(x)|^{\frac{s-1}{2s}}\, \dtext x\\
                &\le \lf(\int_{\mathbb I_t}\frac{|Dg(x)|^s}{|J_g(x)|^{\frac{s-1}{2}}} \, \dtext x \r)^{\frac{1}{s}}\lf(\int_{\mathbb I_t}|J_g(x)|^{\frac{1}{2}}dx\r)^{\frac{s-1}{s}}\, .
\end{split}
\]

\end{proof}

Now, we are ready to prove our main result Theorem \ref{thm:main}.
\subsection{The nonexistence part}
Recall that critical power of inward cusps $ \beta_{\textnormal{cr} }$  is given by the formula~\eqref{eq:betac}. Here we prove that if $\beta \ge \beta_{\textnormal{cr}}$, then there is no homeomorphism $f \colon \mathbb C \to \mathbb C$ of finite distortion with  $f(\dd^\prec_\beta)=\dd$ and $K_f\in \mathscr \cl^{\, p}(\mathbb B_R\setminus\overline{\dd^\prec_\beta})\cap\cl^{\, q}(\dd^\prec_\beta)$ for every $R>2$.
For that suppose that there exists such a homeomorphism. Write
\[s \bydef \min \{p,q\} > 1\, . \]
We will split our argument into two parts. According to Lemma~\ref{lem:epsilon} (we denote $\mathcal J =\{n_k \in \mathbb N \colon k \in \mathbb N \}$) there exists a set $\mathcal J \subset \mathbb N$ and a decreasing sequence $\epsilon_j$ such that $\epsilon_j \to 0$ as $j\to \infty$ and for every $j\in \mathcal J$ we have either
\begin{enumerate}[(i)]
\item\label{c1} $\abs{g (\U_{\epsilon_j})} \le \epsilon_j^2$ or
\item\label{c2} $\abs{g(\U_{\epsilon_j})} \le 5 \abs{g(\U_{\epsilon_{j+1}})}$, $\abs{g (\U_{\epsilon_j})} > \epsilon_j^2$  and $\epsilon_j =2 \epsilon_{j+1}$.
\end{enumerate}
We simplify the notation a little bit and write $\mathbb U_{j}=\mathbb U_{\epsilon_j}$.
In both cases we will integrate  the inequality~\eqref{eq:ineq1} with respect to the variable $t$ and then bound the right hand side by the following basic estimate.
\begin{equation}\label{eq:basicrhs}
\begin{split}
& \lf(\int_{\mathbb U_j}\frac{|Dg(x)|^s}{|J_g(x)|^{\frac{s-1}{2}}}  \, \dtext x\r)^{\frac{2}{s}}\lf(\int_{\mathbb U_j}|J_g(x)|^{\frac{1}{2}}\, \dtext x\r)^{\frac{2(s-1)}{s}} \\
& \; \;  \le  \begin{cases}
C_1(\epsilon_j) \,  \abs{\mathbb U_j}^\frac{p-1}{p} \cdot \abs{g(\mathbb U_j)}^\frac{q-1}{q} \quad & \textnormal{when } q\, , p < \infty   \\
C_2(\epsilon_j) \,   \abs{\mathbb U_j} \cdot \abs{g(\mathbb U_j)}^\frac{q-1}{q}   \quad &  \textnormal{when }  p = \infty \\
C_3(\epsilon_j) \,  \abs{\mathbb U_j}^\frac{p-1}{p} \cdot \abs{g(\mathbb U_j)}   \quad & \textnormal{when } q = \infty  \, .
\end{cases}
\end{split}
\end{equation}
Here the functions $C_1(\epsilon_j)$, $C_2(\epsilon_j)$ and $C_3(\epsilon_j)$ converge to $0$ as $j \to \infty$.
\begin{proof}[Proof of~\eqref{eq:basicrhs}]
 Since $f$ is a mapping of $\mathscr L^s$-distortion  and  $h(0)=f^{-1}(0) \not\in \overline{ \mathbb U_{j}}$ applying Theorem~\ref{thm:reflection} we have
\begin{equation}\label{eq:blah2}
\int_{\mathbb U_j}\frac{|Dg(x)|^s}{|J_g(x)|^{\frac{s-1}{2}}}  \, \dtext x \le  \lf(\int_{g(\U_j)}K^s_f(x)\, \dtext x\r)^{\frac{1}{2}}\cdot\lf(\int_{\U_j} K^s_f(x)\, \dtext x\r)^{\frac{1}{2}} \, .
\end{equation}
Especially, Theorem~\ref{thm:reflection} tells us that $g \in \W^{1,1}_{\loc} (\C, \C)$.  Therefore,   Lemma~\ref{lem:differen} and Lemma~\ref{lem:co-area} give
\begin{equation}
\int_{\U_j} \abs{J_g(x)} \, \dtext x \le \abs{g(\U_j)}  \, .
\end{equation}
This together with H\"older's inequality implies
\begin{equation}\label{eq:blah3}
\int_{\U_j} \abs{J_g(x)}^\frac{1}{2} \, \dtext x \le \abs{\U_j}^\frac{1}{2}\abs{g(\U_j)}^\frac{1}{2}  \, .
\end{equation}
Combining~\eqref{eq:blah2} and~\eqref{eq:blah3} we conclude that

\begin{equation}\label{eq:blah5}
\begin{split}
& \lf(\int_{\mathbb U_j}\frac{|Dg(x)|^s}{|J_g(x)|^{\frac{s-1}{2}}}  \, \dtext x\r)^{\frac{2}{s}}\lf(\int_{\mathbb U_j}|J_g(x)|^{\frac{1}{2}}\, \dtext x\r)^{\frac{2(s-1)}{s}} \\
& \; \; \le     \left(  \int_{g(\U_j)}K^s_f(x) \, \dtext x \cdot    \int_{\U_j}K^s_f(x)\,  \dtext   x \right)^\frac{1}{s}  \left( \abs{\U_j}\cdot \abs{g(\U_j)} \right)^\frac{s-1}{s} \, .
\end{split}
\end{equation}
Recall that $1 < s = \min \{p,q\}<\infty$. Now the claimed inequality~\eqref{eq:basicrhs} follows from the estimate~\eqref{eq:blah5} after applying H\"older's inequality  with
\begin{equation}
\begin{split}
C_1 (\epsilon_j) & \bydef ||K_f||_{\mathscr L^p (\mathbb U_j)} ||K_f||_{\mathscr L^q (g(\mathbb U_j))} \\
 C_2 (\epsilon_j) & \bydef ||K_f||_{\mathscr L^\infty (\mathbb U_j)} ||K_f||_{\mathscr L^q (g(\mathbb U_j))} \\
 C_3 (\epsilon_j) & \bydef ||K_f||_{\mathscr L^p (\mathbb U_j)} ||K_f||_{\mathscr L^\infty (g(\mathbb U_j))} \, .
\end{split}
\end{equation}

\end{proof}

\subsubsection{Case~\eqref{c1}} Recall that in this case we assume that $\abs{g (\U_{j})} \le \epsilon_j^2$.  The homeomorphism $f$ is a mapping of $\mathscr L^s$-distortion,  Lemma~\ref{lem:key} implies that for almost every $0<t<1$ we have
\begin{equation}\ell (g(\mathbb I_t)) \le  \lf(\int_{\mathbb I_t}\frac{|Dg(x)|^s}{|J_g(x)|^{\frac{s-1}{2}}}  \, \dtext x\r)^{\frac{1}{s}}\lf(\int_{\mathbb I_t}|J_g(x)|^{\frac{1}{2}}\, \dtext x\r)^{\frac{s-1}{s}}\, .\end{equation}
Since the curve $g(\mathbb I_t)$ connects the points $(t,t^\beta)$ and $(t,-t^\beta)$ staying in $\dd^\prec_\beta$, the  length of $g(\mathbb I_t)$ is at least $2t$. Therefore,
\begin{equation} 2t \le  \lf(\int_{\mathbb I_t}\frac{|Dg(x)|^s}{|J_g(x)|^{\frac{s-1}{2}}}  \, \dtext x\r)^{\frac{1}{s}}\lf(\int_{\mathbb I_t}|J_g(x)|^{\frac{1}{2}}\, \dtext x\r)^{\frac{s-1}{s}}\, .\end{equation}
Integrating this estimate from $0$ to $\epsilon_j$ with respect to the variable $t$ and applying H\"older's inequality we obtain
\begin{equation}\label{eq:blah1}
\epsilon_j^2 \le    \lf(\int_{\mathbb U_j}\frac{|Dg(x)|^s}{|J_g(x)|^{\frac{s-1}{2}}}  \, \dtext x\r)^{\frac{1}{s}}\lf(\int_{\mathbb U_j}|J_g(x)|^{\frac{1}{2}}\, \dtext x\r)^{\frac{s-1}{s}}\, .
\end{equation}
After squaring this and applying the basic estimate~\eqref{eq:basicrhs} we conclude that
\[\epsilon_j^4 \le \begin{cases}
C_1(\epsilon_j) \,  \abs{\mathbb U_j}^\frac{p-1}{p} \cdot \abs{g(\mathbb U_j)}^\frac{q-1}{q} \quad & \textnormal{when } q\, , p < \infty   \\
C_2(\epsilon_j) \,   \abs{\mathbb U_j} \cdot \abs{g(\mathbb U_j)}^\frac{q-1}{q}   \quad &  \textnormal{when }  p = \infty \\
C_3(\epsilon_j) \,  \abs{\mathbb U_j}^\frac{p-1}{p} \cdot \abs{g(\mathbb U_j)}   \quad & \textnormal{when } q = \infty  \, .
\end{cases}\]
 Now, since $\abs{\U_j} = \frac{2 \epsilon_j^{\beta +1}}{\beta +1} \le \epsilon_j^{\beta +1}$ and $\abs{g(\mathbb U_j)} \le \epsilon_j^2$ we have
\[1 \le \begin{cases}
C_1(\epsilon_j) \,  \epsilon_j^\frac{(\beta- \beta_{\textnormal{cr}} )(pq-q)}{pq} \quad & \textnormal{when } q\, , p < \infty   \\
C_2(\epsilon_j) \,  \epsilon_j^{\beta- \beta_{\textnormal{cr}}}   \quad &  \textnormal{when }  p = \infty \\
C_3(\epsilon_j) \,  \epsilon_j^\frac{(\beta-\beta_{\textnormal{cr}})(p-1)}{p}  \quad & \textnormal{when } q = \infty  \, .
\end{cases}\]
Note that  $C_1 (\epsilon_j)$, $C_2 (\epsilon_j)$ and  $C_3 (\epsilon_j)$ converge to $0$ as $j \to \infty$. Therefore, $\beta < \beta_{\textnormal{cr}}$, this finishes the proof of Theorem~\ref{thm:main} in Case~\eqref{c1}.
\subsubsection{Case~\eqref{c2}} As in the previous case applying  Lemma~\ref{lem:key} for almost every $0<t<1$ we have
\begin{equation}\ell (g(\mathbb I_t)) \le  \lf(\int_{\mathbb I_t}\frac{|Dg(x)|^s}{|J_g(x)|^{\frac{s-1}{2}}}  \, \dtext x\r)^{\frac{1}{s}}\lf(\int_{\mathbb I_t}|J_g(x)|^{\frac{1}{2}}\, \dtext x\r)^{\frac{s-1}{s}}\, .\end{equation}
Now,  we first note that $ 2 \, \ell (g(\mathbb I_t)) \ge  \ell\big(  \partial g(\mathbb U_t) \big)$ and then apply the isoperimetric inequality, Lemma~\ref{lem:isoperimetric} we get
\begin{equation}\abs {g(\mathbb U_t)}^\frac{1}{2} \le  \lf(\int_{\mathbb I_t}\frac{|Dg(x)|^s}{|J_g(x)|^{\frac{s-1}{2}}}  \, \dtext x\r)^{\frac{1}{s}}\lf(\int_{\mathbb I_t}|J_g(x)|^{\frac{1}{2}}\, \dtext x\r)^{\frac{s-1}{s}}\, .\end{equation}
Integrating from $\epsilon_{j+1}$ to $\epsilon_j$ with respect to $t$ we obtain
\[(\epsilon_j -\epsilon_{j+1}) \abs {g(\mathbb U_{j+1})}^\frac{1}{2}  \le  \lf(\int_{\mathbb U_j}\frac{|Dg(x)|^s}{|J_g(x)|^{\frac{s-1}{2}}}  \, \dtext x\r)^{\frac{1}{s}}\lf(\int_{\mathbb U_j}|J_g(x)|^{\frac{1}{2}}\, \dtext x\r)^{\frac{s-1}{s}}\, .
\]
Since by the assumptions of Case~\eqref{c2}$,  \abs{g(\U_{j})} \le 5 \abs{g(\U_{{j+1}})}$ and $\epsilon_j =2 \epsilon_{j+1}$ we have
\[\epsilon_j \abs {g(\mathbb U_{j})}^\frac{1}{2}  \le  10 \lf(\int_{\mathbb U_j}\frac{|Dg(x)|^s}{|J_g(x)|^{\frac{s-1}{2}}}  \, \dtext x\r)^{\frac{1}{s}}\lf(\int_{\mathbb U_j}|J_g(x)|^{\frac{1}{2}}\, \dtext x\r)^{\frac{s-1}{s}}\, .
\]
Combining this with~\eqref{eq:basicrhs} we obtain
\[
\epsilon_j^2 \abs {g(\mathbb U_{j})} \le 100 \cdot  \begin{cases}
C_1(\epsilon_j) \,  \abs{\mathbb U_j}^\frac{p-1}{p} \cdot \abs{g(\mathbb U_j)}^\frac{q-1}{q} \quad & \textnormal{when } q\, , p < \infty   \\
C_2(\epsilon_j) \,   \abs{\mathbb U_j} \cdot \abs{g(\mathbb U_j)}^\frac{q-1}{q}   \quad &  \textnormal{when }  p = \infty \\
C_3(\epsilon_j) \,  \abs{\mathbb U_j}^\frac{p-1}{p} \cdot \abs{g(\mathbb U_j)}   \quad & \textnormal{when } q = \infty  \, .
\end{cases}
\]
Therefore,
\[
\epsilon_j^2  \le 100 \cdot  \begin{cases}
C_1(\epsilon_j) \,  \abs{\mathbb U_j}^\frac{p-1}{p} \cdot \abs{g(\mathbb U_j)}^{-\frac{1}{q}} \quad & \textnormal{when } q\, , p < \infty   \\
C_2(\epsilon_j) \,   \abs{\mathbb U_j} \cdot \abs{g(\mathbb U_j)}^{-\frac{1}{q}}   \quad &  \textnormal{when }  p = \infty \\
C_3(\epsilon_j) \,  \abs{\mathbb U_j}^\frac{p-1}{p}    \quad & \textnormal{when } q = \infty  \, .
\end{cases}
\]

This time  $\abs{\U_j} = \frac{2 \epsilon_j^{\beta +1}}{\beta +1} \le \epsilon_j^{\beta +1}$ and $\abs{g(\mathbb U_j)} > \epsilon_j^2$. Therefore,
\[1 \le 100 \cdot \begin{cases}
C_1(\epsilon_j) \,  \epsilon_j^\frac{(\beta- \beta_{\textnormal{cr}})(pq-q)}{pq} \quad & \textnormal{when } q\, , p < \infty   \\
C_2(\epsilon_j) \,  \epsilon_j^{\beta- \beta_{\textnormal{cr}}}   \quad &  \textnormal{when }  p = \infty \\
C_3(\epsilon_j) \,  \epsilon_j^\frac{(\beta-\beta_{\textnormal{cr}})(p-1)}{p}  \quad & \textnormal{when } q = \infty  \, .
\end{cases}\]
Therefore $\beta < \beta_{\textnormal{cr}}$. This finishes the proof of nonexistence part of Therorem~\ref{thm:main}.

\subsection{The existence part}
In this section, we construct a homeomorphism of finite  distortion $f \colon \C\to\C$ with $f(\dd^\prec_\beta)=\dd$ and $K_f\in \cl^{\, p}(\mathbb B_R\setminus\overline{\dd^\prec_\beta})\cap\cl^{\, q}(\dd^\prec_\beta)$ for every $R>2$, whenever $1\le \beta<\beta_{\textnormal{cr}}$.  Simplifying the  construction we will  replace the unit disk $\mathbb D$ by $\dd^\prec_1$. This causes no loss of generality because  $\dd^\prec_1$ is Lipschitz regular. Indeed, for every Lipschitz domain $\Omega$ there exists a global bi-Lipschitz change of  variables $\Phi \colon \mathbb C \to \mathbb C$ for which $\Phi (\Omega)$ is the unit disk. Therefore, the domains $\dd^\prec_1$ and $\mathbb D$ are bi-Lipschitz equivalent. Especially,  $\dd^\prec_1$ is a quasidisk. Hence we may also  assume the strict inequality  $1<\beta<\beta_{\textnormal{cr}}$ in the construction.

In addition to these we will construct a self-homeomorphism of the unit disk onto itself which  coincide with identity on the boundary. Note that  this causes no loss of generality since $1\pm i \in \dd^\prec_\beta$ and therefore extending the constructed homeomorphism as the identity map to the complement of unit disk. In summary, it suffices to construct a homeomorphism $f \colon \mathbb D \onto \mathbb D$, $f(z)=z$ on $\partial \mathbb D$,  $f(\dd^\prec_\beta)=\dd^\prec_1$ and  $K_f\in\cl^{\, p}(\mathbb D\setminus\overline{\dd^\prec_\beta})\cap\cl^{\, q}(\dd^\prec_\beta)$.  We will  use the polar coordinates $(r,\theta)$ and write  $f \colon \dd\to\dd$ in the form $f(r, \theta)=(\tilde r(r), \tilde\theta(\theta, r))$. Here $\tilde r:[0, 1]\onto[0,1]$ is a strictly increasing function defined by
\begin{equation}\label{eq:rtilde}
\tilde r(r)  \bydef   \begin{cases}
\frac{e}{\exp\lf(\lf(\frac{1}{r}\r)^{\gamma_\beta}\r)} \quad& \textnormal{when } q < \infty   \\
r    \quad & \textnormal{when } q = \infty  \, .
\end{cases}
\end{equation}
The value $\gamma_\beta$ is chosen so that
\begin{equation}
\begin{cases}
\max\lf\{\frac{\beta(p-1)-(p+1)}{p}, 0\r\}<\gamma_\beta<\frac{2}{q} & \textnormal{ when } p<\infty \\
\gamma_\beta=\beta-1  & \textnormal{ when } p=\infty \, .
\end{cases}
\end{equation}
 For every $0<r<1$ we choose $a_r, b_r\in S(0, r)\cap\D\dd^\prec_\beta$ such that $\im a_r>0$ and $\im b_r <0$.  Here and what follows we write $S(0,r)= \partial \mathbb D (0,r)$. Respectively, we choose $\tilde a_{\tilde r(r)}, \tilde b_{\tilde r(r)}\in S(0, \tilde r(r))\cap\D\dd^\prec_1$ such that $\im \tilde a_{\tilde r(r)}>0$ and $\im \tilde b_{\tilde r(r)}<0$. We define the argument function $\tilde\theta(r, \theta)$ so that it  satisfies  the following three properties
\begin{itemize}
\item[(1)]  $f(a_r)=\tilde a_{\tilde r(r)}$ and $f(b_r)=\tilde b_{\tilde r(r)}$.
\item[(2)]  $f$ maps the circular arc $S(0, r)\cap\dd^\prec_\beta$ onto the circular arc $S(0, \tilde r(r))\cap\dd^\prec_1$ linearly as a function of $\theta$.
\item[(3)]   $f$ maps the circular arc $S(0, r)\cap\lf(\dd\setminus\overline{\dd^\prec_\beta}\r)$ onto the circular arc $S(0, \tilde r(r))\cap\lf(\dd\setminus\overline{\dd^\prec_1}\r)$ linearly as a function of $\theta$.
\end{itemize}

\begin{figure}[htbp]
\centering
\includegraphics[width=0.9\textwidth]
{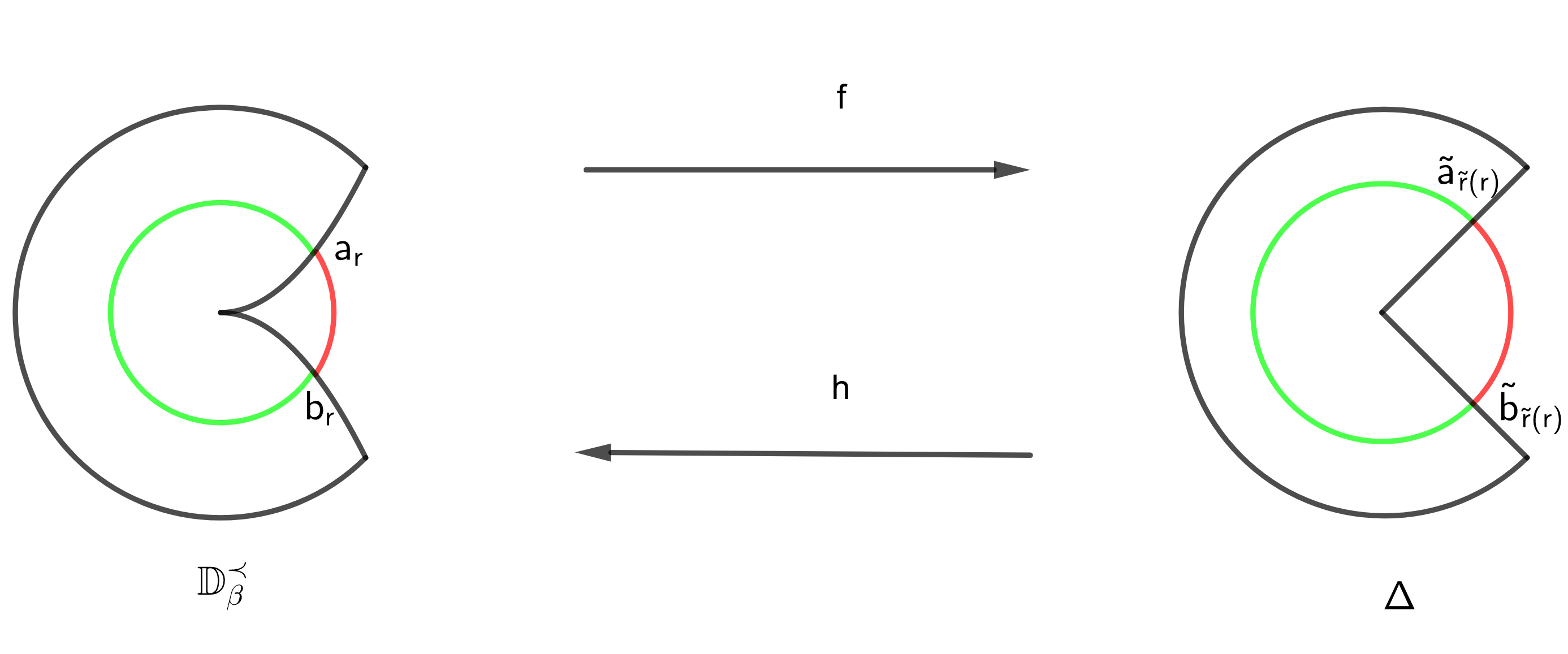}
\end{figure}
We have
\begin{equation}
\mathbb D\,  \cap \,  \dd^\prec_\beta =  \lf\{(r, \theta)\in\C \colon  0<r<1 \; \textnormal{ and } \;  \arctan t^{\beta-1}<\theta<2\pi-\arctan t^{\beta-1}\r\}\nonumber
\end{equation}
and
\begin{equation}
\dd \, \setminus \, \overline{\dd^\prec_\beta}=\lf\{(r, \theta)\in\C \colon  0<r<1 \; \textnormal{ and } \;  -\arctan t^{\beta-1}<\theta<\arctan t^{\beta-1}\r\}.\nonumber
\end{equation}
Here $t>0$ and solves  the equation $t^2+t^{2\beta}=r^2$.  We also have
\begin{equation}
\mathbb D\,  \cap \, \dd^\prec_1 =\lf\{(\tilde r, \tilde \theta)\in\C \colon  0<\tilde r<1 \; \textnormal{ and } \;  \frac{\pi}{4}<\tilde\theta<\frac{7\pi}{4}\r\}\nonumber
\end{equation}
and
\begin{equation}
\dd\setminus\overline{\dd^\prec_1}=\lf\{(\tilde r, \tilde\theta)\in\C \colon 0<\tilde r<1 \; \textnormal{ and } \;  \frac{- \pi}{4}<\tilde\theta<\frac{\pi}{4}\r\}.\nonumber
\end{equation}
Using the polar coordinates  we have
\[
\tilde \theta(\theta, r)  =  \begin{cases}
\frac{3\pi\theta}{4\lf(\pi-\arctan t^{\beta-1}\r)}+\lf(\frac{\pi}{4}-\frac{3\pi\arctan t^{\beta-1}}{4(\pi-\arctan t^{\beta-1})}\r) \quad& \textnormal{when } (r, \theta)\in\dd^\prec_\beta   \\
\frac{\pi\theta}{4\arctan t^{\beta-1}}    \quad & \textnormal{when }  (r, \theta)\in\dd\setminus\overline{\dd^\prec_\beta}.
\end{cases}
\]
For  $(r, \theta)\in\dd$, the differential matrix of $f$ reads as
\begin{equation}
D{f}(r,\theta)
=
\left(
 \begin{array}{ccc}
\frac{\D}{\D r}\tilde r(r) &~~ 0 \\
\tilde r(r)\frac{\D}{\D r}\tilde\theta(r, \theta) &~~ \frac{\tilde r(r)}{r}\frac{\D}{\D\theta}\tilde\theta(r, \theta)\\
\end{array}
\right).\nonumber
\end{equation}
Computing the derivative of radial part $\tilde r (r)$ we have
\begin{equation}\label{eq:diffradial}
\frac{\partial}{\partial r}\tilde r(r)  =  \begin{cases}
\gamma_\beta\lf(\frac{1}{r}\r)^{\gamma_\beta+1}\tilde r(r) \quad& \textnormal{when } q < \infty   \\
1    \quad & \textnormal{when } q = \infty  \, .
\end{cases}
\end{equation}
\subsubsection{Proof of $K_f \in \cl^{\, q}(\dd^\prec_\beta)$}

For $(r, \theta)\in\dd^\prec_\beta$, we have
\begin{center}
$\tilde r(r)\frac{\D}{\D r}\tilde\theta(r, \theta)=\tilde r(r)\frac{\D}{\D r}\lf[\frac{3\pi\theta}{4\lf(\pi-\arctan t^{\beta-1}\r)}+\lf(\pi-\frac{3\pi^2}{4\lf(\pi-\arctan t^{\beta-1}\r)}\r)\r]$
\end{center}
and
\begin{center}
$\frac{\tilde r(r)}{r}\frac{\D}{\D\theta}\tilde\theta(r, \theta)=\frac{\tilde r(r)}{r}\frac{3\pi}{4\lf(\pi-\arctan t^{\beta-1}\r)}\, $.
\end{center}
Since $t>0$ solves the equation $t^2+t^{2\beta}=r^2$, for  $0<r<1$, we have $\frac{\D t}{\D r}\approx 1$ and $0<\arctan t^{\beta-1}<\frac{\pi}{4}$. Here and what follows the notation $A \approx B$ is a shorter form of  two inequalities $A \le c B$ and $B \le c A$ for some positive constant $c$. Therefore, there exists a constant $C>1$ independent of $r$ and $\theta$, such that
\begin{equation}
\abs{\tilde r(r)\frac{\partial}{\partial r}\tilde\theta(r, \theta)}  \le  C \cdot \begin{cases}
\lf(\frac{1}{r}\r)^{\gamma_\beta+1}\tilde r(r) \quad& \textnormal{when } q < \infty   \\
1    \quad & \textnormal{when } q = \infty  \, .
\end{cases}\nonumber
\end{equation}
and
\begin{equation}
\frac{\tilde r(r)}{r}\frac{\partial}{\partial\theta}\tilde\theta(r, \theta)  \approx   \begin{cases}
\frac{\tilde r(r)}{r} \quad& \textnormal{when } q < \infty   \\
1    \quad & \textnormal{when } q = \infty  \, .
\end{cases}\nonumber
\end{equation}
Now, we have
\begin{equation}
K_f(r, \theta) \le C \cdot    \begin{cases}
{r^{-\gamma_\beta}} \quad& \textnormal{when } q < \infty   \\
1   \quad & \textnormal{when } q = \infty  \, .
\end{cases}\nonumber
\end{equation}
for some constant $C>0$.

Since $\gamma_\beta$ is chosen so that $0<\gamma_\beta<\frac{2}{q}$ for $q<\fz$, we have $K_f\in\cl^{\, q}(\dd^\prec_\beta)$.  Also if $q= \infty$, then the distortion function $K_f\in\cl^{\infty}(\dd^\prec_\beta)$, as claimed.
\subsubsection{Proof of  $K_f\in\cl^{\, p}(\dd\setminus\overline{\dd^\prec_\beta})$}

For $(r, \theta)\in\dd\setminus\overline{\dd^\prec_\beta}$, we have
\begin{center}
$\tilde r(r)\frac{\D}{\D r}\tilde\theta(r, \theta)=\tilde r(r)\frac{\D}{\D r}\lf(\frac{\pi\theta}{4\arctan t^{\beta-1}}\r)$
\end{center}
and
\begin{center}
$\frac{\tilde r(r)}{r}\frac{\D}{\D\theta}\tilde\theta(r, \theta)=\frac{\tilde r(r)}{r}\frac{\pi}{4\arctan t^{\beta-1}}$.
\end{center}
Recall that since $t>0$ solves the equation $t^2+t^{2\beta}=r^2$, for $0<r<1$, we have $\frac{\D t}{\D r}\approx 1$. In this case, $-\arctan t^{\beta-1}<\theta<\arctan t^{\beta-1}$, therefore there exists a constant $C>0$ such that
\begin{center}
$\abs{\tilde r(r)\frac{\D}{\D r}\tilde\theta(r, \theta)}\le C\lf(\frac{1}{r}\r)^{\gamma_\beta+1}\tilde r(r)$.
\end{center}
Since
\begin{center}
$\lim\limits_{t\to0^+}\frac{\arctan t^{\beta-1}}{t^{\beta-1}}=1$ and $t<r<2t$,
\end{center}
we have
\begin{center}
$  \frac{\pi}{4\arctan t^{\beta-1}}\frac{\tilde r(r)}{r}\approx  \frac{\tilde r(r)}{r^\beta}$.
\end{center}
Therefore,
\[
K_f(r, \theta) \le\frac{C}{r^{\abs{\beta-\gamma_\beta-1}}} \qquad \textnormal{when } (r, \theta)\in\dd\setminus\overline{\dd^\prec_\beta}
\]
For $p=\infty$, since $\gamma_\beta=\beta-1$, we have $K_f\in\cl^{\, \fz}(\dd\setminus\dd^\prec_\beta)$. For $p<\fz$, $\beta$ is chosen so that $1<\beta<\beta_{\textnormal{cr}}$. When $q<\fz$, $\gamma_\beta$ is chosen so that
\begin{equation}
\max\lf\{\frac{\beta(p-1)-(p+1)}{p}, 0\r\}<\gamma_\beta<\frac{2}{q},\nonumber
\end{equation}
and when $q=\fz$, $\gamma_\beta$ is set to be $0$. Since $\abs{\gamma_\beta+1-\beta}<\frac{2}{p}$ we have
\begin{equation}
\int_{\dd\setminus\overline{\dd^\prec_\beta}}K^p_f(x)\, \dtext x\le \int_0^{2\pi}\int_0^1\frac{1}{r^{p\abs{\beta-\gamma_\beta-1}-1}}\, \dtext r\, \dtext \theta<\fz.\nonumber
\end{equation}

\end{document}